\documentclass[12pt]{article}
\usepackage{latexsym,amsmath,amsthm,verbatim,ifthen,amssymb}
\usepackage[all]{xy}
\usepackage{graphicx}
\usepackage{epsfig}

\newdir{ >}{{}*!/-12pt/@{>}}
\newdir{ (}{!/-3pt/@_{(} } 

\newdir{ )}{!/-3pt/@^{(} } 

\newtheorem{theorem}{Theorem}
\newtheorem{corollary}[theorem]{Corollary}
\newtheorem{lemma}[theorem]{Lemma}
\newtheorem{proposition}[theorem]{Proposition}
\newtheorem{example}[theorem]{Example}
\newtheorem{definition}[theorem]{Definition}
\newtheorem{remark}[theorem]{Remark}

\begin{document}

\title{Formal aspects on parametrized topological complexity and its pointed version}
\author{J.M. Garc\'{\i}a-Calcines}
\maketitle

\begin{abstract}
The notion of parametrized topological complexity, introduced by Cohen, Farber and Weinberger, is extended to fibrewise spaces which are not necessarily Hurewicz fibrations. After exploring some formal properties of this extension we also introduce the pointed version of parametrized topological complexity. Finally we give sufficient conditions so that both notions agree. 
\end{abstract}

\section*{Introduction}
Considering a topological space $X$ as the configuration space of a mechanical system, the motion planning problem consists of giving an algorithm which takes a pair of configurations $(x,y)\in X\times X$ as an input, and produces as an output a path in $X$ starting in $x$ and ending in $y.$ If $X^I$ denotes the space of paths in $X$ equipped with the compact-open topology, then such algorithm is just a section of 
the bi-evaluation fibration $\pi :X^I\rightarrow X\times X,$ $\pi (\gamma )=(\gamma (0),\gamma (1)).$ However, this section cannot be continuous unless the space $X$ is contractible, condition which is almost never true. In order to measure the discontinuity of any motion planner in $X$, M. Farber \cite{F} introduced the notion of topological complexity of a path-connected space $X$, $\mbox{TC}(X).$ It is defined as the sectional category (or Schwarz genus) of $\pi :X^I\rightarrow X\times X.$ In other words, $\mbox{TC}(X)$ is the least nonnnegative integer $n$ such that $X\times X$ admits an cover by $n+1$ open subsets $U_0,U_1,...,U_n$ on each of which $\pi $ admits a continuous local section (called local rule, or local algorithm). From the very beginning this numerical homotopy invariant has become of great interest to many researchers working in the field of algebraic topology. However, in general its computation is fairly hard. 

Recently, D. Cohen, M. Farber and S. Weinberger \cite{C-F-W,C-F-W-2} have introduced a new approach to the theory of motion planning algorithms. In this approach the algorithms involve external conditions which are viewed as parameters taking part of the input. They consider a Hurewicz fibration $p:E\rightarrow B$ which is assumed to be a $1$-equivalence (i.e., each fibre $p^{-1}(b)$ is nonempty and $0$-connected). Here $B$ plays the role of parametrizing the external conditions for the system and for any value $b\in B$ the fibre $p^{-1}(b)$ is the space of achievable configurations of the system under the external condition $b.$ Namely, they define the parametrized topological complexity of $p:E\rightarrow B$ as the sectional category of the associated fibration $\Pi :E^I_B\rightarrow E\times _BE,$ $\gamma \mapsto (\gamma (0),\gamma (1)):$
$$\mbox{TC}[p:E\rightarrow B]:=\mbox{secat}(\Pi )$$
Here, $E^I_B$ denotes the space of paths $\gamma \in E^I$ such that $p\circ \gamma $ is constant, i.e., $\gamma $ lies in a single fibre of $p$. Also, $E\times _BE$ denotes the space of pairs of configurations $(e,e')\in E\times E$ lying in the same fibre, that is, $p(e)=p(e').$ Parametrized topological complexity has proved to be a more sensitive invariant than ordinary topological complexity. In this sense D. Cohen, M. Farber and S. Weinberger have shown that parametrized topological complexity of the problem of collision-free motion of $n$ robots in 3-dimensional space in the presence of $m$ obstacles with unknown a priori position equals $2n+m-1$, in contrast to its standard topological complexity, which is known \cite{F-G-Y} to be just $2n$.

In the work of D. Cohen, M. Farber and S. Weinberger it is assumed that the map considered, $p:E\rightarrow B$, is always a Hurewicz fibration (even more restrictively, a locally trivial fibration) and consequently the map $\Pi $ is also a fibration. However, there are many situations where this condition might not hold. For instance, when we take into account the size of the robots and obstacles in the collision-free motion problem of the above mentioned paragraph. The main goal in this paper is to extend the notion of parametrized topological complexity to maps $E\rightarrow B$, not necessarily fibrations. We are not interested in computing examples, but rather in establishing the theoretical foundations of this extension. The setting in which we will be working throughout this paper is fibrewise homotopy theory. In such a setting the map $\Pi :E^I_B\rightarrow E\times _BE$  is always a fibrewise fibration for any map $p:E\rightarrow B.$ Then we can define fibrewise (or parametrized) topological complexity of the map $p:E\rightarrow B$ as $\mbox{TC}_B(E):=\mbox{secat}_B(\Pi)$; here $\mbox{secat}_B(-)$ stands for fibrewise sectional category, notion that was introduced and studied in \cite{GC}. As we will check, when $p:E\rightarrow B$ is a fibration, the equality $\mbox{secat}_B(\Pi)=\mbox{secat}(\Pi )$ holds true and therefore we have a truly extension of Cohen-Farber-Weinberger's parametrized topological complexity. 

This paper can be also thought as a follow-up to the work established in \cite{GC}, where fibrewise sectional category was introduced. Firstly, in Section 1 we recall some preliminary notions and results about fibrewise homotopy theory. In Section 2, after recalling the notion of fibrewise sectional category we introduce our approach of parametrized topological complexity and establish interesting properties and results. Among them we can mention that it is invariant up to fibrewise homotopy equivalence. Section 3 is fully devoted to introduce a fibrewise pointed version of parametrized topological complexity and to give sufficient conditions so that both invariants agree. In the course of this section the notion of fibrewise locally equiconnected space appears. Such a notion will be crucial for our purposes so we have included a final section, an appendix, where some natural examples and properties of this class of fibrewise spaces are given.

\section{Preliminaries: Fibrewise homotopy theory}

We begin by recalling some notation and results about fibrewise 
homotopy theory. For more details we refer the reader to \cite{C-J} or \cite{GC}.

Let $B$ be a fixed topological space. A fibrewise space over $B$ consists of a pair
$(X,p_X)$ where $X$ is a topological space and $p_X:X\rightarrow B$ is a map. 
The map $p_X$ is usually called the projection of the fibrewise space. When there is no
risk of ambiguity the pair $(X,p_X)$ will be denoted as $X$ and just called fibrewise space.
Given $X$ and $Y$ fibrewise spaces, a fibrewise map (over $B$)
from $X$ to $Y$ consists of an ordinary map $f:X\rightarrow Y$ satisfying
$p_Y\circ f=p_X$
$$\xymatrix{
{X} \ar[rr]^f \ar[dr]_{p_X} & & {Y} \ar[dl]^{p_Y} \\
 & {B} & }$$
We shall denote as $\mathbf{Top}_B$ the resulting category of fibrewise spaces and fibrewise maps
over $B$.  In this category, $B$ with the identity is the final object. The
initial object is $\emptyset $ (with the unique projection map to $B$). In this category, if $X$ and $Y$ 
are fibrewise spaces, then the fibrewise
product of $X$ and $Y$ is $$X\times _BY=\{(x,y)\in
X\times Y:p_X(x)=p_Y(y)\}$$ the pullback of $p_X$ and $p_Y.$ This is indeed the categorical product in 
$\mathbf{Top}_B$ of the objects $X$ and $Y$.
 
We denote by $I$ the closed unit interval $[0,1]$ with the topology induced by the usual one in $\mathbb{R}.$
The \emph{fibrewise cylinder} of a fibrewise space $X$ is just the
product space $X\times I$ together with the composite ${X\times
I}\stackrel{pr}{\longrightarrow }X\stackrel{p_X}{\longrightarrow
}B$ as the projection. We will denote by $I_B(X)$
the fibrewise cylinder of $X$. The notion of \emph{fibrewise
homotopy} $\simeq _B$ between fibrewise maps comes straightforwardly as
well as the notion of \emph{fibrewise homotopy equivalence}. 

If $X$ is a fibrewise space consider the pullback in
the category $\mathbf{Top}$ of topological spaces and maps: 
$$\xymatrix{
{P_B(X)} \ar[r] \ar[d] & {X^I} \ar[d]^{p_X^I} \\
{B} \ar[r]_c & {B^I} }$$ 
Here $X^I$ (and $B^I$) denotes the free path-space provided 
with the cocompact topology and $p_X^I$ is the obvious map induced 
by precomposing with $p_X$. Besides $c:B\rightarrow B^I$ is the map that carries any $b\in B$ to the
constant path $c_b$ in $B^I$. Thus, $P_B(X)$ has the expresion
$$P_B(X)=B\times _{B^I}X^I=\{(b,\alpha )\in B\times X^I:c_b=p_X\circ \alpha \}$$
\noindent with projection $P_B(X)\rightarrow B$, $(b,\alpha )\mapsto b,$ the base change of $p_X^I$ in this pullback. 
$P_B(X)$ is called the \emph{fibrewise cocylinder} of $X$ (or
\emph{fibrewise free path space} of $X$). 

\begin{remark}\label{prim}
Observe that $P_B(X)$ can be also described as the space of all paths $\alpha :I\rightarrow X$ such 
that the path $p_X\circ \alpha $ is constant, i.e., paths lying in single fibre of $X.$ 
This description is given in \cite{C-F-W}, where the notation used is $X_B^I$ for the fibrewise cocylinder.
Also note that $P_B(X)$ is fibrewise homotopy equivalent to $X$. Indeed, the fibrewise map
$$\gamma _X:X\rightarrow P_B(X),\hspace{8pt}x\mapsto (p_X(x),c_x)$$
\noindent is a fibrewise homotopy equivalence with a homotopy inverse $\gamma '_X:P_B(X)\rightarrow X$ defined as 
$\gamma '_X(b,\alpha ):=\alpha (0)$.
\end{remark}

The fibrewise cylinder and fibrewise cocylinder constructions give rise to functors
$$I_B,P_B:\mathbf{Top}_B\rightarrow \mathbf{Top}_B$$
Associated to the functor $I_B$ there are straightforwardly defined natural
transformations $i_0,i_1:X\rightarrow I_B(X)$ and $\rho
:I_B(X)\rightarrow X$. Similarly, associated
to $P_B$ there are natural transformations
$d_0,d_1:P_B(X)\rightarrow X$ and $c:X\rightarrow P_B(X)$. Moreover, $(I_B,P_B)$ is an
\emph{adjoint pair} in the sense of Baues (see \cite[p.29]{B}).
A \emph{fibrewise fibration} is a fibrewise map
$p:E\rightarrow Y$ satisfying the Homotopy Lifting
Property with respect to any fibrewise space, i.e., 
given any commutative diagram of solid arrows in $\mathbf{Top}_B$
$$\xymatrix{
{Z} \ar[r]^f \ar[d]_{i_0} & {E} \ar[d]^p \\
{I_B(Z)} \ar[r]_H \ar@{.>}[ur] & {Y} }$$
\noindent the dotted arrow exists in $\mathbf{Top}_B$ making commutative all the diagram. 
As the functor $I_B$ is left adjoint to the functor $P_B$ it is plain to check that, actually,
fibrewise fibrations are precisely the internal fibrations in $\mathbf{Top}_B$ with respect to $P_B.$
Therefore $\mathbf{Top}_B$ together with $P_B$ the fibrewise cocylinder is a $P$-category in the 
sense of Baues \cite[p.31, Prop (4.6)]{B}. 

If $p:E\rightarrow Y$ is any fibrewise map such that it is an ordinary
Hurewicz fibration, then $p$ is a fibrewise fibration. In
general, the converse is not true. For instance, if $X$ is a
fibrewise space, then $p_X:X\rightarrow B$ is always a fibrewise
fibration, but $p_X$ need not be a Hurewicz fibration.
 
\begin{remark}
From the general axiomatic theory of a $P$-category, if $X$ is any fibrewise space, then the fibrewise map 
$\Pi =(d_0,d_1)$, defined as
$$\Pi :P_B(X)\rightarrow X\times _B X,\hspace{6pt}(b,\alpha )\mapsto (\alpha (0),\alpha (1))$$
\noindent is always a fibrewise fibration, which is not necessarily a Hurewicz fibration. 
Nevertheless, we point out that when the projection $p_X:X\rightarrow B$ is a Hurewicz fibration, then
one can check that $\Pi :P_B(X)\rightarrow X\times _B X$ is also a Hurewicz fibration.
\end{remark}

A fibrewise map $j:A\rightarrow X$ over $B$ is said to be a \emph{fibrewise
cofibration} if it satisfies the Homotopy Extension Property, that is, for 
any fibrewise map $f:X\rightarrow
Y$ and any fibrewise homotopy $H:I_B(A)\rightarrow Y$ such that
$H\circ i_0=f\circ j,$ there exists a fibrewise homotopy
$\widetilde{H}:I_B(X)\rightarrow Y$ such that $\widetilde{H}\circ i_0=f$
and $\widetilde{H}\circ I_B(j)=H$
$$\xymatrix{
{A} \ar[d]_j \ar[r]^{i_0} & {I_B(A)} \ar[d]^H
\ar@/^2pc/[ddr]^{I_B(j)} & \\ {X} \ar[r]^f \ar@/_2pc/[drr]_{i_0} &
{Y} &
\\ & & {I_B(X)} \ar@{.>}_{\widetilde{H}}[ul]}$$

As known, fibrewise cofibrations are cofibrations in the usual
sense. Therefore we can consider, without loss of generality, that the fibrewise
cofibrations are inclusions $A\hookrightarrow X.$
The pair $(X,A)$ is then called \emph{fibrewise cofibred pair}. Similarly, fibrewise cofibrations are precisely 
the internal cofibrations in $\mathbf{Top}_B$ with respect to $I_B,$
so $\mathbf{Top}_B$ together with $I_B$ the fibrewise cylinder is an $I$-category in the 
sense of Baues \cite[p.31, Prop (4.6)]{B}.  

In actual fact, if $fib_B,$ $\overline{cof}_B$ and $he_B$ denote the classes of
fibrewise fibrations, closed fibrewise cofibrations (equivalently,
closed fibrewise cofibred pairs) and fibrewise homotopy
equivalences, respectively, then the category $\mathbf{Top}_B$ together with the classes of
$\overline{cof}_B,$ $fib_B$ and $he_B$ has an IP-category
structure in the sense of Baues \cite{B}. More is true,  
the category $\mathbf{Top}_B$ with 
the classes $\overline{cof}_B,$ $fib_B$ and $he_B$ has 
a proper closed model category structure in the sense of Quillen \cite{GC,M-S}.

\section{Parametrized topological complexity}
In this section we introduce the notion of parametrized (or fibrewise) topological complexity of a fibrewise space over $B$. Then we see that it fulfills interesting properties. We previously need to introduce fibrewise sectional category, given in \cite{GC}.

\subsection{Fibrewise sectional category}

We begin this section by recalling the notion of fibrewise sectional category. 
For more details we refer the reader to \cite{GC}.

\begin{definition}\cite{GC} Given $f:E\rightarrow X$ a fibrewise map over $B$, a subset $U\subseteq X$ 
is said to be \emph{fibrewise sectional} if there exists a fibrewise map $s:U\rightarrow E$ 
such that $f\circ s\simeq _B inc,$ where $inc:U\hookrightarrow X$ denotes the inclusion map. 
In other words, the following diagram in $\mathbf{Top}_B$ is commutative up to fibrewise homotopy:
$$\xymatrix{ {U} \ar@{^(->}[rr]^{inc} \ar[dr]_s & & {X} \\
 & {E} \ar[ur]_f & }$$
The \emph{fibrewise sectional category} of $f$, $\mbox{secat}_B(f),$ is the minimal number $k$ 
such that $X$ admits a cover $\{U_i\}_{i=0}^k$ by $k+1$ fibrewise sectional open subsets. If there is no 
such $k$, then we set $\mbox{secat}_B(f)=\infty .$
\end{definition}

\begin{remark}
Observe that, when $f:E\rightarrow X$ is a fibrewise fibration, then we can replace commutativity up to fibrewise
homotopy with strict commutativity in the definition of fibrewise sectional subset.
\end{remark}

A particular case of fibrewise sectional category is that of \emph{fibrewise unpointed L.S. category} in the sense 
of Iwase-Sakai \cite{I-S,I-S-2}. Recall that a \emph{fibrewise pointed space over } (or ex-space) $B$ is just a 
fibrewise space $X$ together with a fibrewise map $s_X:B\rightarrow X$; this explicitly means that $p_X\circ s_X=1_B$). 
With our notion, the fibrewise unpointed L.-S. category of $X$ is just 
$$\mbox{cat}^*_B(X):=\mbox{secat}_B(s_X)$$

The relationship between fibrewise sectional category and fibrewise unpointed L.-S. category is given by the following property.
Consider $f:E\rightarrow X$ a fibrewise pointed map; that is, $E$ and $X$ are pointed and $f$ is a fibrewise map satisfying
$f\circ s_X=s_Y.$ Then we have that
$$\mbox{secat}_B(f)\leq \mbox{cat}^*_B(X)$$
Moreover, if $E$ is fibrewise contractible (i.e., a shrinkable space), then the equality $\mbox{secat}_B(f)=\mbox{cat}^*_B(X)$ holds.

Another interesting property is that $\mbox{secat}_B(-)$ is invariant up to fibrewise homotopy.

\begin{proposition}\label{inv-secfib}
Consider the following diagram in $\mathbf{Top}_B$, which is commutative up to fibrewise homotopy:
$$\xymatrix{
{E} \ar[d]_f \ar[rr]^{\alpha }_{\simeq _B} & & {E'} \ar[d]^{f'} \\
{X} \ar[rr]_{\beta }^{\simeq _B} & & {X'}
}$$ \noindent and $\alpha ,\beta $ being fibrewise homotopy equivalences. Then $\mbox{secat}_B(f)=\mbox{secat}_B(f').$
\end{proposition}

\begin{proof}
Suppose first that $X=X'$ and $\beta =1_X$. 
In this case, if $U\subseteq X$ is an open subset with $s:U\rightarrow E$ a fibrewise
homotopy section of $f$ then $\alpha \circ s:U\rightarrow E'$ is a fibrewise homotopy section of $f'.$ 
This proves that $\mbox{secat}_B(f')\leq \mbox{secat}_B(f).$ Similarly, considering the homotopy inverse of $\alpha $ we obtain
the reverse inequality and therefore $\mbox{secat}_B(f)\leq \mbox{secat}_B(f').$ 

Now suppose that $E=E'$ and $\alpha =1_E.$ If $U\subseteq E'$ is an open subset with 
$s:U\rightarrow E$ a homotopy section of $f'$, then we define $V:=\beta ^{-1}(U)$ and $s'$ the composite 
$V\stackrel{\beta |_V}{\longrightarrow }U\stackrel{s}{\longrightarrow }E.$ As
$\beta \circ f\circ s'\simeq _B inc_U \circ (\beta |_V)=\beta \circ inc_V$ we have that $f\circ s'\simeq _B inc_V.$
This proves $\mbox{secat}_B(f)\leq \mbox{secat}_B(f').$ A similar reasoning with the homotopy inverse of $\beta $ 
gives the equality $\mbox{secat}_B(f)=\mbox{secat}_B(f').$ 

Now consider the general case. Using the particular cases above we obtain
$\mbox{secat}_B(f)=\mbox{secat}_B(\beta \circ f)=\mbox{secat}_B(f'\circ \alpha )=\mbox{secat}_B(f').$
\end{proof}

An interesting property of fibrewise sectional category is given by its relationship with the ordinary sectional category. 
From now on, by a \emph{fibrant} fibrewise space over $B$ we will mean a fibrewise space $X$ in which its projection 
$p_X:X\rightarrow B$ is an ordinary Hurewicz fibration.  

\begin{proposition}\cite[Th.2.10]{GC} \label{fib-nfib}
Let $f:E\rightarrow X$ be a fibrewise map between fibrant spaces over $B.$ Then
$\mbox{secat}_B(f)=\mbox{secat}(f).$
\end{proposition}

\begin{remark}
In Proposition \ref{fib-nfib} we can also consider \emph{weak fibrant} spaces over $B$, meaning that
the projection is a weak fibration (i.e., a Dold fibration). Recall that a Dold fibration may be defined as a map which is
fibrewise homotopy equivalent to a Hurewicz fibration.
\end{remark}

\subsection{Parametrized topological complexity}
Recall that given any fibrewise space $X$ over $B$, there is defined a fibrewise fibration
$\Pi :P_B(X)\rightarrow X\times _B X$ given as $(b,\alpha )\mapsto (\alpha (0),\alpha (1))$.

\begin{definition}
The parametrized (or fibrewise) topological complexity of a fibrewise space $X\stackrel{p_X}{\rightarrow }B$ is defined as
$$\mbox{TC}_B(X):=\mbox{secat}_B(\Pi _X:P_B(X)\rightarrow X\times _B X)$$
\end{definition}

\begin{remark}
We observe that, when $X$ is fibrant, that is, the projection $p_X:X\rightarrow B$ is an ordinary Hurewicz fibration, then 
both $P_B(X)$ and $X\times _B X$ are also fibrant. Therefore $\mbox{TC}_B(X)=\mbox{secat}(\Pi _X)$ by Proposition \ref{fib-nfib} above. In other words, this notion agrees with the one given
by Cohen, Farber and Weinberger \cite{C-F-W,C-F-W-2}, provided the fibrewise space is fibrant.
\end{remark}

\begin{remark}
Also observe than, taking into account the commutative diagram
$$\xymatrix{
{X} \ar[rr]^{\Delta _X} \ar[dr]_{\gamma _X}^{\simeq _B} & & {X\times _B X} \\
 & {P_B(X)} \ar[ur]_{\Pi _X} &
}$$
\noindent and using Proposition \ref{inv-secfib} we can equivalently define $\mbox{TC}_B(X)=\mbox{secat}_B(\Delta _X).$ Here $\Delta _X$ denotes the diagonal map and $\gamma _X$ the fibrewise homotopy equivalence defined in Remark \ref{prim}.

\end{remark}
\bigskip

Now we see that this is a fibrewise homotopy invariant.

\begin{proposition}
Let $X,X'$ be fibrewise spaces over $B.$ If $X\simeq _B X'$ are fibrewise homotopy equivalent, 
then $\mbox{TC}_B(X)=\mbox{TC}_B(X').$
\end{proposition}

\begin{proof}
Consider $f:X\rightarrow X'$ a fibrewise homotopy equivalence. Taking into account that the map $\gamma _X$, defined in Remark
\ref{prim}, is actually a natural transformation, we have that $P_B(f):P_B(X)\rightarrow P_B(X')$ is a fibrewise homotopy equivalence.
Moreover, by the gluing lemma in a $P$-category \cite{B} applied to the following diagram in $\mathbf{Top}_B$
$$\xymatrix{
{X} \ar[rr]^{p_X} \ar[d]^{\simeq _B}_{f} & & {B} \ar[d]_{1_B}^{\simeq _B} & & {X} \ar[d]^f_{\simeq _B} \ar[ll]_{p_X} \\
{X'} \ar[rr]_{p_{X'}} & & {B} & & {X'} \ar[ll]^{p_{X'}}
}$$ \noindent we also have that $f\times _B f:X\times _BX\rightarrow X'\times _BX'$ is a fibrewise homotopy equivalence. The results 
follows by just applying Proposition \ref{inv-secfib} to the commutative diagram
$$\xymatrix{
{P_B(X)} \ar[d]_{\Pi _X} \ar[rr]_{\simeq _B}^{P_B(f)} & & {P_B(X')} \ar[d]^{\Pi _{X'}} \\
{X\times _B X} \ar[rr]^{\simeq _B}_{f\times _B f} & & {X'\times _B X'}
}$$ 
\end{proof}

As a consequence of the result above we have that $\mbox{TC}_B(X)=0$ provided that $X$ is fibrewise contractible (i.e., $X\simeq _B B$). 
Observe that $\mbox{TC}_B(B)=\mbox{secat}_B(1_B)=0.$ For the next result we need $X$ to be fibrewise pointed.

\begin{corollary}
Let $X$ be a fibrewise pointed space over $B$. Then $\mbox{TC}_B(X)=0$ if, and only if, $X$ is fibrewise contractible.
\end{corollary}

\begin{proof}
Certainly, by the previous proposition, if $X\simeq _B B$, then $\mbox{TC}_B(X)=\mbox{TC}_B(B)=0.$ Now suppose that $\mbox{TC}_B(X)=0$, that is, we have a strictly commutative
triangle in $\mathbf{Top}_B$
$$\xymatrix{
{X\times _B X} \ar[rr]^{1} \ar[dr]_{\sigma } & & {X\times _BX} \\
 & {P_B(X)} \ar[ur]_{\Pi } &
}$$ 
Then the fibrewise map $\sigma $ is necessarily of the form $\sigma (x,y)=(p_X(x),\overline{\sigma }(x,y))$ where 
$\overline{\sigma }(x,y)$ is a path in $X$ from $x$ to $y$ such that $p_X\circ \overline{\sigma }(x,y)$ is the constant path in 
$p_X(x)(=p_X(y))$.

Now, the projection $p=p_{P_B(X)}:P_B(X)\rightarrow B$ (considered as a fibrewise map) has a homotopy fibrewise homotopy
inverse the fibrewise map $q:B\rightarrow P_B(X),$ defined as $q(b):=(b,c_{s_X(b)})$; here $s_X:B\rightarrow X$ is the section of $p_X.$ Indeed, $p\circ q=1_B$ and $q\circ p\simeq _B1_{P_B(X)}$ through the fibrewise homotopy
$$H:I_B(P_B(X))\rightarrow P_B(X),\hspace{10pt}(b,\alpha ,t)\mapsto (b,\overline{H}(b,\alpha ,t))$$ \noindent where 
$\overline{H}(b,\alpha ,t)(s):=\overline{\sigma }(\alpha (st),s_X(b))(1-t).$ This proves that $P_B(X)\simeq _B B.$ As $X\simeq _B P_B(X)$ (see Remark \ref{prim}) we conclude the result.
\end{proof}

\bigskip
Let $X$ be a fibrewise \emph{pointed} space over $B$. In this case we have that 
$$\Pi :P_B(X)\rightarrow X\times _B X$$
\noindent is indeed a fibrewise pointed map over $B$. The section $s_{P_B(X)}:B\rightarrow P_B(X)$ is given by 
$b\mapsto (b,c_{s_X(b)})$, where $c_{s_X(b)}$ denotes the constant path at the point $s_X(b).$ Moreover, the section
$s_{X\times _B X}:B\rightarrow X\times _B X$ is given by $b\mapsto (s_X(b),s_X(b)).$
Therefore, we have by \cite[Prop.2.3]{GC} that $\mbox{TC}_B(X)\leq \mbox{cat}_B^*(X\times _BX).$
More is true:

\begin{proposition}
Let $X$ be a fibrewise pointed space over $B$. Then
$$\mbox{cat}_B^*(X)\leq \mbox{TC}_B(X)\leq \mbox{cat}_B^*(X\times _BX)$$
\end{proposition}

\begin{proof}
It only remains to check the first inequality. Take $U\subseteq X\times _BX$ an open subset with $\sigma :U\rightarrow P_B(X)$ a local section
of $\Pi :P_B(X)\rightarrow X\times _B X$. Then, $\sigma $ has necessarily the expression
$$\sigma (x,y)=(p_X(x),\overline{\sigma }(x,y))$$ Here the map $\overline{\sigma }:U\rightarrow X^I$ satisfies $\overline{\sigma }(x,y)(0)=x, \overline{\sigma }(x,y)(1)=y$ and $p_X\circ \overline{\sigma }(x,y)=C_{p_X(x)}$ for all $(x,y)\in U.$ Now, on the open subset of $X$ given as
$V:=\{x\in X:(x,(s_X\circ p_X)(x))\in U\}$ we define the fibrewise homotopy $H:I_B(V)\rightarrow X$ as 
$H(x,t):=\overline{\sigma }(x,s_X(p_X(x)))(t)$. This proves that $V$ is fibrewise categorical, that is, we have a commutativity, up to fibrewise homotopy, of the triangle
$$\xymatrix{
{V} \ar@{^(->}[rr] \ar[dr]_{p_X|_V} & & {X} \\
 & B \ar[ur]_{s_X} & 
}$$ The results follows using this reasoning with open covers.
\end{proof}

\bigskip
Now consider a continuous map $\lambda :B'\rightarrow B.$ Given any fibrewise space $X$ over $B$ the pullback of $p_X$ and $\lambda $
gives rise to a fibrewise space over $B'$:
$$\xymatrix{
{X'} \ar[r] \ar[d]_{p_{X'}} & {X} \ar[d]^{p_X} \\
{B'} \ar[r]_{\lambda } & {B}
}$$ Here $X'=\{(b',x)\in B'\times X: \lambda (b')=p_X(x)\}$ with the obvious projection over $B'.$ Besides, for any fibrewise map over $B$, $f:X\rightarrow Y$ there is an induced
fibrewise map over $B'$, $f':X'\rightarrow Y'$, defined as
$f'(b',x):=(b',f(x))$. 
This construction gives rise to a well-defined functor $$\lambda ^*:\mathbf{Top}_B\rightarrow \mathbf{Top}_{B'}$$ \noindent where $\lambda ^*(X):=X'$ and $\lambda ^*(f):=f'.$

\begin{proposition}
Let $f:E\rightarrow X$ be a fibrewise map over $B$ and $\lambda :B'\rightarrow B$ a map. Then
$\mbox{secat}_{B'}(\lambda ^*f)\leq \mbox{secat}_B(f)$.
\end{proposition}

\begin{proof}
Let $U\subseteq X$ be an open subset together with $\sigma :U\rightarrow E$ a fibrewise homotopy section (over $B$) of $f:E\rightarrow X.$ Then there is a fibrewise homotopy over $B$, $H:I_B(U)\rightarrow X$, satisfying $H(x,0)=x$ and $H(x,1)=\sigma (x)$, for all $x\in U.$ We consider on the open subset of $\lambda ^*(X)$
$$V:=\{(b',x)\in B'\times U:\lambda (b')=p_X(x)\}$$
\noindent the fibrewise map over $B'$, $\sigma ':V\rightarrow \lambda ^*(E)$, defined as $\sigma '(b',x):=(b',\sigma (x)).$ Then, it is straightforward to check that the fibrewise homotopy over $B'$, $H':I_{B'}(V)\rightarrow \lambda ^*(E)$, given as $H'(b',x,t):=(b',H(x,t))$ gives a fibrewise homotopy section (over $B'$), i.e., the following diagram, which is commutative up to fibrewise homotopy over $B'$: 
$$\xymatrix{
{V} \ar@{^(->}[rr] \ar[dr]_{\sigma '} & & {\lambda ^*(X)} \\
 & {\lambda ^*(E)} \ar[ur]_{\lambda ^*(f)} & 
}$$
Considering this argument for open covers we obtain the result.
\end{proof}

\begin{corollary}
Let $\lambda :B'\rightarrow B$ be a map and consider $X$ a fibrewise space over $B$. 
Then $\mbox{TC}_{B'}(\lambda ^*(X))\leq \mbox{TC}_{B}(X).$
\end{corollary}

\begin{proof}
One has just to take into account that $\lambda ^*(\Pi _X)$ is exactly the fibrewise fibration over $B'$
$$\Pi _{\lambda ^*(X)}:P_{B'}(\lambda ^*(X))\rightarrow \lambda ^*(X)\times _{B'}\lambda ^*(X)$$ 
\noindent and then consider the proposition above.
\end{proof}

\begin{corollary}
Let $X$ be a fibrewise space over $B$. If $B'\subseteq B$ and we take $X'=(p_X)^{-1}(B')$ and $p_{X'}:X'\rightarrow X$ the restriction map, then
$\mbox{TC}_{B'}(X')\leq \mbox{TC}_B(X).$
\end{corollary}

\begin{proof}
Just observe that $X'$ is the pullback of $p_X$ along the inclusion $i:B'\hookrightarrow B$. In other words, $X'=i^*(X).$
\end{proof}

In particular, if we consider $F_b=(p_X)^{-1}(b)$ for an arbitrary point $b\in B$, then we have that
$$\mbox{TC}(F_b)\leq \mbox{TC}_B(X)$$ \noindent where $\mbox{TC}(-)$ denotes the usual topological complexity.

\section{The fibrewise pointed case}

We shall denote by $\mathbf{Top}(B)$ the category of fibrewise pointed spaces and fibrewise pointed maps over $B$. This is a pointed category where $B$ is the zero object. Fibrewise pointed spaces are also called \emph{ex-spaces} in the literature.

For any fibrewise pointed space $X$ we can consider its  fibrewise pointed cylinder as the topological pushout 
$$\xymatrix{
{B\times I} \ar[r]^{pr} \ar[d]_{s_X\times id} & {B} \ar[d] \\
{X\times I} \ar[r]  & {I_B^B(X)} }$$ \noindent together with the obvious projection induced by the pushout property. Then we can consider  fibrewise pointed homotopy between fibrewise pointed maps, that will be denoted as $\simeq ^B_B.$ The notion of fibrewise pointed homotopy equivalence also comes naturally. We can also consider $P_B(X)$ as a fibrewise pointed space, where the section $B\rightarrow P_B(X)$ is defined as $b\mapsto (b,C_{s_X(b)}).$ We can consider \emph{(closed)  fibrewise pointed cofibration} and 
\emph{fibrewise pointed fibrations} (i.e. the internal (closed) cofibrations and fibrations in $\mathbf{Top}(B)$). These are defined by the
natural Homotopy Extension Property and the Homotopy Lifting
Property in $\mathbf{Top}(B),$ respectively. Then we have that $\mathbf{Top}(B)$ has the structure of an $I$-category and a $P$-category in the sense of Baues (\cite[p.31]{B}). We are particularly interested in a particular class of fibrewise pointed spaces.

A \emph{fibrewise well-pointed} space is a fibrewise pointed space
$X$ in which the section $s_X:B\rightarrow X$ is a closed
fibrewise cofibration. Let $\mathbf{Top}_w(B)$ denote the full
subcategory of $\mathbf{Top}(B)$ consisting of fibrewise well-pointed 
spaces.

\begin{proposition}\cite{GC} \label{importante2}
$\mathbf{Top}_w(B)$ is closed under the pullbacks of fibrewise
pointed maps which are fibrewise fibrations. Similarly,
$\mathbf{Top}_w(B)$ is closed under the pushouts of fibrewise
pointed maps which are closed fibrewise cofibrations. Moreover
if $f:X\rightarrow Y$ is a fibrewise pointed map between
fibrewise well-pointed spaces over $B,$ then,

\begin{enumerate}
\item[(i)] $f$ is a fibrewise pointed fibration if and only if $f$ is a fibrewise
fibration;

\item[(ii)] If $f$ is a closed map, then $f$ is a fibrewise pointed cofibration if and only if $f$ is a fibrewise
cofibration;

\item[(iii)] $f$ is a  fibrewise pointed homotopy equivalence if and only if $f$ is a fibrewise homotopy
equivalence.
\end{enumerate}

\end{proposition}

Now we recall from \cite{GC} the notion of fibrewise pointed sectional category:
Let $f:E\rightarrow X$ be a fibrewise map over $B$ and consider an
open subset $U$ of $X$ containing $s_X(B)$. Then $U$ is said to be
\emph{fibrewise pointed sectional} if there exists 
$s:U\rightarrow E$ in $\mathbf{Top}(B)$ such that the following
triangle commutes up to  fibrewise pointed homotopy
$$\xymatrix{ {U} \ar@{^(->}[rr]^{in} \ar[dr]_s & & {X} \\
 & {E} \ar[ur]_f & }$$ The \emph{fibrewise pointed sectional category} of $f,$
$\mbox{secat}^B_B(f)$, is the minimal number $n$ such
that $X$ admits an open cover $\{U_i\}_{i=0}^n$ constituted by
fibrewise pointed sectional subsets. When such $n$ does not exist
then $\mbox{secat}^B_B(f)=\infty .$ 

\begin{remark}
When $f:E\rightarrow X$ is a fibrewise pointed fibration, then we can replace commutativity up to fibrewise
pointed homotopy with strict commutativity in the definition of fibrewise pointed sectional subset.
\end{remark}

Given $X$ a  fibrewise pointed space over $B$ we already know that  
$\Pi _X:P_B(X)\rightarrow X\times _B X$ is a  fibrewise pointed map. More is true, from the general theory of a $P$-category, we
have that $\Pi _X$ is, actually, a fibrewise pointed fibration.

\begin{definition}
Let $X$ be a fibrewise pointed space over $B$. We define the fibrewise pointed topological complexity of $X$ as
$$\mbox{TC}^B_B(X):=\mbox{secat}^B_B(\Pi _X)$$
\end{definition}

\begin{remark}
Clearly, we have the inequality $\mbox{TC}_B(X)\leq \mbox{TC}^B_B(X),$ for any fibrewise pointed space over $B.$ 
\end{remark}

\begin{example}\label{ejemplo}
Consider the category $\mathbf{Top}^*$ of pointed spaces and pointed maps. This is the category $\mathbf{Top}(B)$ of fibrewise pointed spaces
over $B$ when $B=*.$ There is a notion of \emph{pointed topological complexity} for a given pointed space $X.$ Indeed, if $x_0\in X$ denotes the base point, then $\mbox{TC}^*(X)$ is the least integer $k$ (or infinity) such that $X\times X$ admits an open cover $\{U_i\}_{i=0}^k$ where each $U_i$ contains $(x_0,x_0)$ and there exists a pointed local section $s_i:U_i\rightarrow X^I$ of $\pi :X^I\rightarrow X\times X$ (i.e., $s_i(x_0,x_0)=c_{x_0}$ is the constant path at $x_0$ and $\pi \circ s_i$ equals the inclusion $U_i\hookrightarrow X\times X$). By Proposition \ref{pointed} below, $\mbox{TC}^*(-)$ is invariant up to pointed homotopy equivalence which satisfies analogue properties in the pointed setting as in the classical case. 
\end{example}

Using similar arguments to that used for $\mbox{TC}_B(-)$ it is not difficult to check that $\mbox{TC}^B_B(-)$ satisfies analogue properties to those given for $\mbox{TC}_B(-)$. We leave to the reader the details of the proof of the following proposition.

\begin{proposition}\label{pointed}${}$
\begin{enumerate}
\item  Fibrewise pointed topological complexity may also be described as $$\mbox{TC}_B^B(X):=\mbox{secat}^B_B(\Delta _X)$$
\noindent where $\Delta _X:X\rightarrow X\times _B X$ stands for the diagonal map.

\item If $X\simeq ^B_B Y$ are  fibrewise pointed homotopy equivalent, then $\mbox{TC}_B^B(X)=\mbox{TC}_B^B(Y).$

\item Let $X$ be a fibrewise pointed space over $B$. Then $\mbox{TC}_B^B(X)=0$ if, and only if, $X\simeq ^B_B B.$

\item For $X$ any fibrewise pointed space over $B$ the following inequalities hold true:
$$\mbox{cat}^B_B(X)\leq \mbox{TC}_B^B(X)\leq \mbox{cat}^B_B(X\times _BX)$$ \noindent Here, $\mbox{cat}^B_B(-)$ denotes James-Morris fibrewise L.-S. category \cite{J-M}.

\item Let $\lambda :B'\rightarrow B$ be a continuous map. Then, the pullback construction gives rise to a functor $\lambda ^*:\mathbf{Top}(B')\rightarrow \mathbf{Top}(B)$. Such a functor satisfies
$$\mbox{TC}^{B'}_{B'}(\lambda ^*(X))\leq \mbox{TC}^{B}_{B}(X)$$ \noindent for any fibrewise pointed space $X$ over $B$.
\end{enumerate}
\end{proposition}

 No we want to see how close $\mbox{TC}_B(-)$ and $\mbox{TC}_B^B(-)$ are. In order to do this, we will use the following result, whose proof can be found in \cite{GC}.

\begin{theorem}\cite[Th.4.1]{GC}\label{first-condition}
Let $f:E\rightarrow X$ be a fibrewise pointed map in
$\mathbf{Top}_w(B)$ between normal spaces, or a closed fibrewise
cofibration with $X$ normal. Then
$$\mbox{secat}_B(f)\leq \mbox{secat}_B^B(f)\leq \mbox{secat}_B(f)+1$$
\end{theorem}

We will consider fibrewise spaces $X$ in which the diagonal map $\Delta :X\rightarrow X\times _B X$ is a closed fibrewise cofibration. The class of such fibrewise spaces will be called \emph{fibrewise locally equiconnected spaces}. 

\begin{definition}
Let $X$ be a fibrewise space over $B$. Then $X$ is said to \emph{fibrewise locally equiconnected} when $\Delta _X:X\rightarrow X\times _B X$ is a closed fibrewise cofibration. 
\end{definition}

Of course, every locally equiconnected space (i.e., the diagonal map is a closed classical cofibration) is a fibrewise locally equiconnected space over $B=*.$ The class of fibrewise locally equiconnected spaces is not as restrictive as it may appear at first glance. We give some natural examples in the final appendix of this paper. Moreover, if $X\in \mathbf{Top}(B)$ is a fibrewise pointed space over $B$ and $X$ is fibrewise locally equiconnected, then necessarily $X\in \mathbf{Top}_w(B)$ is a fibrewise well-pointed space (see Corollary \ref{wp} below).

As a consequence of Theorem \ref{first-condition} we have our first result for the relationship between fibrewise topological complexity and its pointed counterpart.

\begin{corollary}
Let $X\in \mathbf{Top}(B)$ be a fibrewise pointed space over $B$. If $X$ is fibrewise locally equiconnected and $X\times _BX$ is normal, then
$$\mbox{TC}_B(X)\leq \mbox{TC}^B_B(X)\leq \mbox{TC}_B(X)+1$$
\end{corollary}

\begin{proof}
Just apply Theorem \ref{first-condition} to $\Delta _X:X\rightarrow X\times _BX$ and take into account that $\mbox{TC}_B(X)=\mbox{secat}_B(\Delta _X)$, $\mbox{TC}_B^B(X)=\mbox{secat}_B^B(\Delta _X)$. Observe that, necessarilly $X$ and $X\times _BX$ are fibrewise well-pointed  spaces. 
\end{proof}

\begin{remark}
Observe that the condition on $X\times _BX$ being normal is not that restrictive. For instance, when $B$ is Hausdorff and $X$ is metrizable, then this condition holds.
\end{remark}

For our next result we will use the following theorem.
If $X$ is a fibrewise pointed space over $B,$ then we say that $X$
is \emph{cofibrant} when the section $s_X$ is an ordinary closed cofibration
in $\mathbf{Top}.$ Observe the difference with the notion of being
well-pointed, in which $s_X$ must be a closed fibrewise cofibration. 
As any fibrewise cofibration is, in particular, an ordinary cofibration, 
we obviously have that any fibrewise well-pointed space is cofibrant.

\begin{theorem}\cite[Th.4.4]{GC}
Let $f:E\rightarrow X$ be a fibrewise pointed map between fibrewise pointed
fibrant and cofibrant spaces over $B.$ Suppose, in addition, that
$B$ is a CW-complex and the following conditions are satisfied:
\begin{enumerate}
\item[(i)] $f:E\rightarrow X$ is a $k$-equivalence ($k\geq 0$);

\item[(ii)] $\mbox{dim}(B)<(\mbox{secat}_B(f)+1)(k+1)-1.$
\end{enumerate}
\noindent Then $\mbox{secat}_B(f)=\mbox{secat}_B^B(f).$
\end{theorem}

As an immediate consequence of this theorem we have:

\begin{proposition}\label{fundam}
Let $X$ be a fibrewise pointed space over $B$ which is fibrewise locally equiconnected and fibrant.
Consider $B$ a CW-complex and suppose the following conditions are satisfied:
\begin{enumerate}
\item[(i)] $\Delta _X:X\rightarrow X\times _BX$ is a $k$-equivalence ($k\geq 0$).

\item[(ii)] $\mbox{dim}(B)<(\mbox{TC}_B(X)+1)(k+1)-1.$
\end{enumerate}
Then $\mbox{TC}_B(X)=\mbox{TC}_B^B(X).$
\end{proposition}

\begin{proof}
Observe that both $X$ and $X\times _BX$ are fibrewise well-pointed and therefore cofibrant. 
Moreover, as $X$ is fibrant, then $X\times _BX$ also is. Then, we have just to apply the theorem above to
$f=\Delta _X.$
\end{proof}

\begin{example}
We consider the case of $\mathbf{Top}^*$ given in Example \ref{ejemplo}. Then, if $X$ is any path-connected, locally equiconnected space (for instance, a connected CW-complex or a connected ANR space), then we have that topological complexity and pointed topological complexity agree on this space:
$$\mbox{TC}(X)=\mbox{TC}^*(X)$$
Indeed, observe that if $X$ is contractible, then it is contractible in the pointed sense since $X$ is well-pointed; therefore, $\mbox{TC}(X)=0=\mbox{TC}^*(X).$ When $X$ is not contractible, conditions (i) and (ii) apply since $\Delta _X:X\rightarrow X\times X$ is a 0-equivalence and $B=*.$

Observe than, in general it is not true that $\mbox{TC}(X)=\mbox{TC}^*(X)$. Consider $X=\bigcup
_{n=0}^{\infty }L_n$, where $L_n=\{(t,\frac{t}{n}):t\in
[0,1]\}$ is the segment in $\mathbb{R}^2$ joining $(0,0)$ with $(1,\frac{1}{n})$ for
$n\geq 1$ and $L_0=[0,1]\times \{0\}$. 
 Clearly $X$ is a
contractible space and therefore $\mbox{TC}(X)=0.$ On the other hand, if we take $x_0=(1,0)$ as a base point for $X$, then $\mbox{TC}^*(X)\neq 0$. Otherwise we would have $s:X\times X\rightarrow X^I$ a continuous section of the
path fibration $\pi :X^I\rightarrow X\times X$ satisfying
$s(x_0,x_0)=c_{x_0}$. The sequence $x_n=(1,\frac{1}{n})$ converges to $x_0$ so, by continuity,  
$s(x_n,x_0)$ must converge to $s(x_0,x_0)=c_{x_0}.$ This is not possible,
as $s(x_n,x_0)$ is a path from $x_n$ to $x_0$
that must pass through the origin $(0,0)$, for all $n\in \mathbb{N}.$ An easier argument is to check that $(X,x_0)$ is not contractible in the pointed sense, and therefore $\mbox{TC}^*(X)$ cannot be $0.$

\end{example}

\medskip
We end our study with a short discussion about when $\Delta _X:X\rightarrow X\times _BX$ is a $k$-equivalence. Consider the following (homotopy) commutative diagram in $\mathbf{Top}$:
$$\xymatrix{
{Y} \ar[d]^{\beta } \ar[r] & {X} \ar[d]^{\alpha } & {Z} \ar[d]^{\gamma } \ar[l] \\
{Y'} \ar[r] & {X'} & {Z'} \ar[l] }$$
Then it is well-known that if $\alpha $ is an $(n+1)$-equivalence and $\beta ,\gamma $ are $n$-equivalences, then the map induced between the corresponding homotopy pullbacks $Y\times ^h_XZ\rightarrow Y'\times ^h_{X'}Z'$ is an $n$-equivalence. Therefore, assuming that $X$ is fibrant (i.e., $p_X:X\rightarrow B$ is an ordinary fibration) we deduce from the commutative diagram 
$$\xymatrix{
{X} \ar[d]^{1_X} \ar[r]^{1_X} & {X} \ar[d]^{p_X} & {X} \ar[d]^{1_X} \ar[l]_{1_X} \\
{X} \ar[r]_{p_X} & {B} & {X} \ar[l]^{p_X} }$$
\noindent that $\Delta _X:X\rightarrow X\times _BX$ is a $k$-equivalence provided that $p_X:X\rightarrow B$ is a $(k+1)$-equivalence.

As a consequence, we may restate Proposition \ref{fundam} above replacing condition (i) with $p_X:X\rightarrow B$ being a $(k+1)$-equivalence. In particular we have the following result:

\begin{corollary}
Let $X$ be a fibrewise pointed space over a connected CW-complex $B$. Suppose that $X$ is fibrewise locally equiconnected and $p_X:X\rightarrow B$ is a fibration with nonempty path-connected fibre. If $\mbox{dim}(B)<\mbox{TC}_B(X)$, then $\mbox{TC}_B(X)=\mbox{TC}_B^B(X).$
 \end{corollary}

\section{Appendix: Fibrewise locally equiconnected spaces.}

The aim of this auxiliary section is to give examples and some interesting results on fibrewise locally equiconnected spaces, among them their relationship with fibrewise uniformly locally contractible spaces. The main reference in this section is \cite{C-J} and we will consider similar cofibration-like techniques to those given in \cite{Wr}. We begin by recalling this latter class of fibrewise spaces.


\begin{definition}\label{fulc} Let $X$ be a fibrewise space over $B$. Then $X$ is said to be fibrewise uniformly locally contractible if there exist $W$ an open neighborhood of the diagonal in $X\times _B X$ and a fibrewise homotopy $G:I_B(W)\rightarrow X$ such that $G(x,y,0)=x,$ $G(x,y,1)=y$ for all $(x,y)\in W$, and $G(x,x,t)=x$ for all $x\in X$ and $t\in I.$
\end{definition}

By Proposition 5.16 in \cite{C-J} any \emph{fibrewise ANR} space (in particular any \emph{fibrewise ENR} space) is a fibrewise uniformly locally contractible space. For the notion of fibrewise ANR (resp. ENR) space see definition 5.5 in \cite{C-J}. Basic examples of fibrewise ANR spaces are bundles of normed vector spaces and fibre bundles whose fibres are ANR spaces. Moreover, basic examples of fibrewise ENR spaces are finite-dimensional real vector bundles over $B$ and fibre bundles whose fibres are ENR spaces. 

Another interesting example of fibrewise ANR space is given by any Hurewicz fibration $p_X:X\rightarrow B$ where $X$ is an ordinary ANR space (Prop. 5.20 in \cite{C-J}).

Now we see the following interesting property. Recall that any fibrewise cofibration $A\rightarrow X$ is, actually, an embedding, and therefore we have a fibrewise cofibred pair $(X,A)$. Fibrewise cofibred pairs are characterized by the existence of \emph{fibrewise Str{\o}m structures} (see \cite{C-J} or \cite{GC}). Namely, a fibrewise pair $(X,A)$ is a fibrewise cofibred pair if, and only if,  
$(X,A)$ admits a fibrewise Str{\o}m structure, that is, a pair $(\varphi ,H)$ consisting of:
\begin{enumerate}

\item[(i)] A map $\varphi :X\rightarrow I$ satisfying $A\subseteq \varphi
^{-1}(\{0\});$

\item[(ii)] A fibrewise homotopy $H:I_B(X)\rightarrow X$
satisfying $H(x,0)=x,$ $H(a,t)=a$ for all $x\in X,$ $a\in A,$
$t\in I$, and $H(x,t)\in A$ whenever $t>\varphi (x).$
\end{enumerate}
If $A$ is closed, then necessarily $A=\varphi ^{-1}(\{0\}).$

\medskip
\begin{proposition}
Any fibrewise locally equiconnected space is fibrewise uniformly locally contractible.
\end{proposition}

\begin{proof}
Let $X$ be a fibrewise locally equiconnected space and consider $(\varphi ,H)$ a fibrewise Str{\o}m structure for the fibrewise cofibred pair $(X\times _BX,\Delta _X(X))$. We define the open subset $W:=\varphi ^{-1}([0,1[)$. Clearly, $W$ is an open neighborhood of $\Delta _X(X)$ in $X\times _BX.$ We also define the fibrewise homotopy $G:I_B(U)\rightarrow X$ as
$$G(x,y,t):=\begin{cases} 
pr_1(H(x,y,2t)), & 0\leq t\leq \frac{1}{2} \\
pr_2(H(x,y,2-2t)) & \frac{1}{2}\leq t\leq 1
\end{cases}$$
\noindent where $pr_1,pr_2:X\times _BX\rightarrow X$ are the corresponding projections. Then, one can straightforwardly check that $G(x,y,0)=x,$ $G(x,y,1)=y$ for all $(x,y)\in W$, and $G(x,x,t)=x$ for all $x\in X$ and $t\in I.$
\end{proof}

Now our aim is to give a converse of the proposition above. We need a previous lemma.

\begin{lemma}\label{clave}
Let $(X,A)$ a fibrewise pair where $A$ is closed in $X$. Assume that there exist a map $\pi :X\rightarrow [0,1]$, an open subset $U\subseteq X$ with $A=\pi ^{-1}(\{0\})\subseteq \pi ^{-1}(]0,1])\subseteq U$ (in particular $U$ is a halo of $A$) and a fibrewise homotopy $G:I_B(U)\rightarrow X$ such that $G(x,0)=x,$ $G(a,t)=a$ and $G(x,1)\in A,$ for all $x\in X,$ $a\in A$, $t\in I.$ Then $(X,A)$ is a fibrewise cofibred pair.
\end{lemma}

\begin{proof}
We use the fact that $(X,A)$ is a fibrewise cofibred pair if, and only if, $X\times \{0\}\cup A\times I$ is a fibrewise retract of $I_B(X)$ (Prop 4.1 in \cite{C-J}). The fibrewise retraction $r:I_B(X)\rightarrow X\times \{0\}\cup A\times I$ is given by the following description:
\begin{enumerate}
\item[(i)] $r(x,t)=(x,0),$ if $\pi (x)=0;$

\item[(ii)] $r(x,t)=(G(x,2\pi (x)t),0),$ if $0<\pi (x)\leq \frac{1}{2};$

\item[(iii)] $r(x,t)=(G(x,\frac{t}{2(1-\pi (x))}),0),$ if $\frac{1}{2}\leq \pi (x)<1$ and $0\leq t\leq 2(1-\pi (x));$

\item[(iv)] $r(x,t)=(G(x,1),t-2(1-\pi (x))),$ if $\frac{1}{2}\leq \pi (x)<1$ and $2(1-\pi (x))\leq t\leq 1;$

\item[(v)] $r(x,t)=(x,t),$ if $\pi (x)=1.$

\end{enumerate}
\end{proof}

For our result we consider the notion of \emph{fibrewise Hausdorff} space. This is just a fibrewise space $X$ in which the diagonal map $\Delta _X:X\rightarrow X\times _BX$ is closed. Equivalently, for each $b\in B$ and each pair of distinct points $x,x'$ lying in the same fibre over $b$, there exist disjoint neighborhoods of $x,x'$ in $X.$ In particular, if $X$ is Hausdorff, then it is fibrewise Hausdorff.

\begin{theorem}
Let $X$ be a fibrewise space over $B$ where $X$ is fibrewise Hausdorff and $X\times _BX$ is perfectly normal. Then $X$ is fibrewise locally equiconnected if, and only if, $X$ is fibrewise uniformly locally contractible.
\end{theorem}

\begin{proof}
If $X$ is a fibrewise uniformly locally contractible space then we can consider $W$ an open neighborhood of $\Delta _X(X)$ in $X\times _BX$ and a fibrewise homotopy $H:I_B(W)\rightarrow X$ satisfying the conditions in Definition \ref{fulc}. As $X\times _BX$ is perfectly normal and $\Delta _X(X)$, $(X\times _BX)\setminus W$ are disjoint closed subsets in $X\times _BX$, we can find a map $\pi :X\times _BX\rightarrow [0,1]$ such that $\Delta _X(X)=\pi ^{-1}(\{1\})$ and $W=\pi ^{-1}(]0,1]).$ The fibrewise homotopy $G:I_B(W)\rightarrow X\times _BX$, defined as $G(x,y,t):=(H(x,y,t),y)$, satisfies the conditions in Lemma \ref{clave} above. Therefore, $\Delta _X:X\rightarrow X\times _BX$ is a closed fibrewise cofibration.
\end{proof}

\begin{remark}
The conditions of $X$ being fibrewise Hausdorff and $X\times _BX$ perfectly normal are not restrictive. For instance, when $X$ is metrizable such two conditions hold.  
\end{remark}

\begin{corollary}
Any fibrewise ANR is fibrewise locally equiconnected. In particular, the following classes of fibrewise spaces are fibrewise locally equiconnected:
\begin{enumerate}
\item[(i)] Fibrewise ENR spaces.

\item[(ii)] Hurewicz fibrations $p_X:X\rightarrow B$ where $X$ is an ordinary ANR space.

\item[(iii)] Fibre bundles whose fibres are ordinary ANR spaces.

\item[(iv)] Bundles of normed vector spaces.

\end{enumerate}
\end{corollary}

We end this appendix by giving some interesting properties on fibrewise locally equiconnected spaces.

\begin{proposition}
Let $X$ be a fibrewise locally equiconnected space. Suppose a map $f:X\rightarrow [0,1]$ such that $A:=f^{-1}(\{0\})$ is a fibrewise retract of $f^{-1}([0,1[).$ Then $(X,A)$ is a fibrewise cofibred pair.
\end{proposition}

\begin{proof}
Consider $r:f^{-1}([0,1[)\rightarrow A$ a fibrewise retraction and let $(\varphi, H)$ be a fibrewise Str{\o}m structure for the fibrewise cofibred pair $(X\times _BX, \Delta _X(X)).$ We define $U:=\varphi ^{-1}([0,1[)$ and a fibrewise homotopy $G:I_B(U)\rightarrow X$ by 
$$G(x,y,t):=\begin{cases}
pr_1(H(x,y,2t)), & 0\leq t\leq \frac{1}{2} \\
pr_2(H(x,y,2-2t)), & \frac{1}{2}\leq t\leq 1
\end{cases}$$
\noindent where $pr_1,pr_2$ denote the corresponding projections $X\times _BX\rightarrow X.$ We also define a map $\varphi _f:X\rightarrow [0,1]$ as 
$$\varphi _f(x):=\begin{cases}
\max \{f(x),\varphi (x,r(x))\}, & f(x)<1 \\
1, & f(x)=1
\end{cases}$$
We have that $\varphi _f^{-1}(\{0\})=A.$ Moreover, the fibrewise homotopy $$G_f:I_B(\varphi _f^{-1}([0,1[))\rightarrow X$$ \noindent 
defined as $G_f(x,t):=G(x,r(x),t)$ satisfies all the requirements in Lemma \ref{clave} above. Therefore, $(X,A)$ is a closed fibrewise cofibred pair. 
\end{proof}

\begin{corollary}
Let $X$ be a fibrewise locally equiconnected space and consider $e:X\rightarrow X$ a fibrewise map such that $e\circ e=e.$ Then, $(X,e(X))$ is a closed fibrewise cofibred pair.
\end{corollary}

\begin{proof}
If $(\varphi ,H)$ is a fibrewise Str{\o}m structure for $(X\times _BX,\Delta _X(X),$ then define $f:X\rightarrow [0,1]$ as $f(x):=\varphi (x,e(x)).$ Then we have that $e(X)=f^{-1}(\{0\})$ which is a fibrewise retract of $f^{-1}([0,1[)$ through the fibrewise retraction $r:f^{-1}([0,1[)\rightarrow e(X)$ defined as $r(x):=e(x).$
\end{proof}

\begin{corollary}\label{wp}
Let $X$ be a fibrewise locally equiconnected space and $A\subseteq X.$ If $A$ if a fibrewise retract of $X$, then $(X,A)$ is a closed fibrewise cofibred pair. 
\end{corollary}

\begin{proof}
Denote as $r:X\rightarrow A$ the fibrewise retraction and $i:A\hookrightarrow X$ the inclusion map. Then we have that $e=i\circ r:X\rightarrow X$ is a fibrewise map satisfying $e\circ e=e.$
\end{proof}

\begin{remark}
Observe that, as a consequence of this last result, we have that any fibrewise pointed space $X\in \mathbf{Top}(B)$ is well-pointed (i.e., the section $s_X$ is a closed fibrewise cofibration) as long as $X$ is a fibrewise locally equiconnected space over $B$.
\end{remark}


\begin{thebibliography}{99}

\bibitem{B} {\sc H. J. Baues}. {\em Algebraic Homotopy}.
Cambridge Studies in Advanced Maths., 15. Camb. Univ. press, 1989.



\bibitem{C-F-W} {\sc D.C. Cohen, M. Farber, S. Weinberger}. 
Topology of parametrized motion planning algorithms. arXiv:2009.06023v1  

\bibitem{C-F-W-2} {\sc D.C. Cohen, M. Farber, S. Weinberger}. 
Parametrized topological complexity of collision-free motion planning in the plane. arXiv:2010.09809v1  

\bibitem{C-L-O-T} {\sc O. Cornea, G. Lupton, J. Oprea and D. Tanr\'{e}}.
{\em Lusternik-Schnirelmann category. Math. Surverys and
Monographs}, vol. {\bf 103}, AMS, 2003

\bibitem{C-J} {\sc M. Crabb and I. James}. {\em Fibrewise Homotopy Theory.}
Springer Monographs in Mathematics. Springer-Verlag London, Ltd.,
London, 1998.









\bibitem{F} {\sc M. Farber}. Topological complexity of motion planning. \emph{Discrete Comput.
Geom.} \textbf{29} (2003), 211-221.

\bibitem{F2} {\sc M. Farber}. Instabilities of robot motion. \emph{Topology Appl.}
\textbf{140} (2004), 245-266.

\bibitem{F-G-Y} {\sc M. Farber, M. Grant, S. Yuzvinsky}. Topological complexity of collision free motion planning algorithms in the presence
of multiple moving obstacles, in: Topology and Robotics, 75--83, Contemp. Math., vol. 438, Amer. Math. Soc., Providence, RI, 2007.



\bibitem{GC} {\sc Garcia-Calcines, Jose M.} Whitehead and Ganea constructions for 
fibrewise sectional category. Topology Appl. {\bf 161} (2014), 215–234.




\bibitem{I-S} {\sc N. Iwase, M. Sakai}. Topological complexity is
a fibrewise LS category. \emph{Topology Appl.} \textbf{157}
(2010), 10-21.

\bibitem{I-S-2} {\sc N. Iwase, M. Sakai}.
Erratum to ``Topological complexity is a fibrewise L-S category".
\emph{Topology Appl.} \textbf{59} (2012), 2810-2813



\bibitem{J2} {\sc I.M. James}. Introduction to fibrewise homotopy
theory. Handbook of Algebraic Topology. North-Holland, Amsterdam,
1995, 169-194.

\bibitem{J-M} {\sc I.M. James, J.R. Morris}. Fibrewise category.
\emph{Proc. Roy. Soc. Edinburgh.} Sect. A \textbf{119} (1991)
177-190.



\bibitem{M-S} {\sc J.P. May, J. Sigurdsson}.
{\em Parametrized Homotopy Theory. Math. Surverys and Monographs},
vol. {\bf 132}, AMS, 2006.


\bibitem{Sch} {\sc A. Schwarz}. \emph{The genus of a fiber space}. A.M.S. Transl. 55 (1966),
49-140

\bibitem{Wr} {\sc G. Warner}. {\em Topics in Topology and Homotopy Theory}. arXiv:2007.02467v1 


\end{thebibliography}
\end{document}